\documentclass[12pt,reqno]{amsart}
\usepackage{amsmath,amssymb,amsthm}
\usepackage{graphicx}
\nonstopmode \numberwithin{equation}{section}
\newtheorem{theorem}{Theorem}[section]
\newtheorem{definition}{Definition}
\newtheorem{remark}{Remark}[section]
\newtheorem{conj}{Conjecture}
\newtheorem{lemma}{Lemma}[section]
\newtheorem{corollary}{Corollary}[section]

\newtheorem{proposition}{Proposition}[section]
\begin{document}
\title{ON GENERALIZED CES\`ARO STABLE FUNCTIONS}
\author{
Priyanka Sangal
}
\address{Department of Mathematics, IIT Roorkee}
\email{sangal.priyanka@gmail.com, priyadma@iitr.ac.in}

\author{
A. Swaminathan
$^{\ast}$}

{\thanks{$^{\ast}$Corresponding author}}
\address{
Department of  Mathematics  \\
Indian Institute of Technology, Roorkee-247 667,
Uttarkhand,  India
}
\email{swamifma@iitr.ac.in, mathswami@gmail.com}

\bigskip

\maketitle

\begin{abstract}
The notion of Ces\`aro stable function is generalized by introducing Ces\`aro mean of type
$(b-1;c)$ which give rise to a new concept of generalized Ces\`aro stable function.
As an application of generalized Ces\`aro stable functions we also prove for a convex
function of order $\lambda\in[1/2,1)$, its Ces\`aro mean of type $(b-1;c)$ is
close-to-convex of order $\lambda$. Further two conjectures are also posed in the
direction of generalized Ces\`aro stable function. Some particular cases of these conjectures
are also discussed.
\end{abstract}


\pagestyle{myheadings}\markboth{Priyanka Sangal and A. Swaminathan}
{On Generalized Ces\`aro Stable Functions}


\section{preliminaries}

Let $b+1>c>0$ and $0<\mu<1$. Define the sequence $\{c_k\}$ as
\begin{align}\label{eqn:def-coeff}
c_{2k}=c_{2k+1}=d_k= \frac{B_{n-k}}{B_n} \frac{(\mu)_k}{k!}, \quad k=0,1,2,\ldots
\end{align}
where $B_0=1$ and $B_{k}=\frac{(b)_k}{(c)_k}\frac{1+b-c}{b}$ for $k\geq 1$.

This sequence was used in \cite{sangal-swaminathan-positivity-alpha-beta}
to obtain the positivity of the
trigonometric cosine sums.
\begin{theorem}\cite{sangal-swaminathan-positivity-alpha-beta}
\label{thm:cosine-positivity-sigma-bc}
Let the coefficient $\{c_k\}$ be given as in \eqref{eqn:def-coeff}.
Then for $b\geq c>0$ and $n\in\mathbb{N}$
\begin{align*}
\sum_{k=0}^n c_k \cos k\theta>0 \quad \mbox{for} \quad \mu\leq \mu_0' \mbox{ and } 0<\theta<\pi,
\end{align*}
where $\mu_0'$ is the solution of
\begin{align*}
\int_{0}^{3\pi/2} \frac{\cos t}{t^{1-\mu}}\left(1-\frac{2t}{3\pi}\right)^{b-c} dt=0.
\end{align*}
\end{theorem}
The positivity of sine sums analogous to Theorem \ref{thm:cosine-positivity-sigma-bc}
is also given in \cite{sangal-swaminathan-positivity-alpha-beta}.
\begin{theorem}\cite{sangal-swaminathan-positivity-alpha-beta}
\label{thm:sine-positivity-sigma-bc}
Let the coefficient $\{c_k\}$ be given as in \eqref{eqn:def-coeff}.
Then for $b\geq c>0$, $n\in\mathbb{N}$ and $0<\theta<\pi$ the following inequalities hold.
\begin{align*}
\sum_{k=1}^{2n+1} c_k \sin k\theta &>0 \quad \mbox{for} \quad \mu\leq \mu_0',\\
\sum_{k=1}^{2n} c_k \sin k\theta &>0 \quad \mbox{for} \quad \mu\leq \left(\frac{1+b}{c}\right)-\frac{1}{2}.
\end{align*}
\end{theorem}
Note that for $b=1$ and $c=1$, $c_k$ given in \eqref{eqn:def-coeff} reduces to
$\gamma_k$ given by Vietoris \cite{vietoris-1958}.
\begin{align*}
\gamma_0=\gamma_1=1\quad \gamma_{2k}=\gamma_{2k+1}=\frac{(1/2)_k}{k!}, k\geq 1.
\end{align*}
Clearly Theorem \ref{thm:cosine-positivity-sigma-bc} and Theorem \ref{thm:sine-positivity-sigma-bc}
are further development of the following theorem given by Vietoris \cite{vietoris-1958},
by choosing $a_k=\gamma_k$.
\begin{theorem}\cite{vietoris-1958}\label{thm:vietoris}
Let $\{a_k\}_{k=0}^{\infty}$ be a non-increasing sequence of non-negative real numbers
such that $a_0>0$ and satisfying
\begin{align*}
2k a_{2k}\leq (2k-1)a_{2k-1},\quad  k\geq1,
\end{align*}
then for all positive integers $n$ and $\theta\in(0,\pi)$, we have
\begin{align*}
\sum_{k=1}^n a_k\sin{k\theta}>0 \hbox{  and  }  \sum_{k=0}^n a_k\cos{k\theta}>0.
\end{align*}
\end{theorem}
Vietoris \cite{vietoris-1958} observed that these two inequalities for the special case in which $a_k=\gamma_k$
where the sequence $\gamma_k$ is defined as above.
Several generalizations of Theorem \ref{thm:vietoris} can be found in the literature. For example, see \cite{brown-dai-wang-2007-ext-viet-ramanujan,
koumandos-2007-ext-viet-ramanujan, saiful-2012-stable-results-in-math,
sangal-swaminathan-positivity-alpha-beta}.
As an application of positive trigonometric sums, Ruscheweyh and Salinas
\cite{ruscheweyh-salinas-2000-AnnMarie} introduced the concept of stable functions.
Due to its wide significance, the generalization of Theorem \ref{thm:vietoris}
is of much interest. For the recent development
in this direction see \cite{sangal-swaminathan-positivity-alpha-beta} and the references therein. In \cite{sangal-swaminathan-positivity-alpha-beta}, the sequence $\{c_k\}$ given below is considered which is generalization of the sequence
$\{\gamma_k\}$ considered by Vietoris' \cite{vietoris-1958}.

In \cite{sangal-swaminathan-positivity-alpha-beta} the applications of
Theorem \ref{thm:cosine-positivity-sigma-bc} and
 Theorem \ref{thm:sine-positivity-sigma-bc} in finding the location of zeros of
 a class of trigonometric polynomials is discussed. Some new inequalities
 related to Gegenbauer polynomials are also given in \cite{sangal-swaminathan-positivity-alpha-beta}.
  It is of interest to interpret Theorem \ref{thm:cosine-positivity-sigma-bc} and
  Theorem \ref{thm:cosine-positivity-sigma-bc} in the context of geometric
 function theory. For this purpose, we recall some concepts and definitions.

The set of
analytic functions in the unit disc $\mathbb{D}:=\{z:|z|<1\}$ is denoted by $\mathcal{A}$ and the
set of all one-to-one (univalent) functions in $\mathbb{D}$ is denoted by $\mathcal{S}$.
Let $\mathcal{A}_0$ and $\mathcal{A}_1$ are the subset of $\mathcal{A}$ with normalization $f(0)=0,f'(0)=1$
and $f(0)=1$ respectively.
The following subclasses of $\mathcal{S}$ are useful for further discussion.
Let $\mathcal{S}^{\ast}(\alpha)$, $0\leq\alpha<1$, be the class of starlike functions of order $\alpha$,
$f\in\mathcal{A}$ satisfying $\mathrm{Re}\left(\frac{zf'(z)}{f(z)}\right)>\alpha$ and $\mathcal{C}(\alpha)$, $0\leq\alpha<1$
be the class of convex function of order $\alpha$, satisfying $\mathrm{Re}\left(1+\frac{zf''(z)}{f(z)}\right)>\alpha$, for
$z\in\mathbb{D}$. If we take $\alpha=0$, these two subclasses reduce to starlike and
convex class denoted by $\mathcal{S}^{\ast}$ and $\mathcal{C}$ respectively.
The relation between these two subclasses is given by Alexander transformation i.e.
$f\in\mathcal{C}(\alpha)\Longleftrightarrow zf'\in\mathcal{S}^{\ast}(\alpha)$. One another important subclass
$\mathcal{K}(\alpha)$ be the class of all close-to-convex functions $f\in\mathcal{A}$
with respect to a starlike function $g(z)\in\mathcal{S}^{\ast}$ if
$\mathrm{Re}\, e^{\iota\gamma}\left(\frac{zf'(z)}{g(z)}\right)>\alpha,\gamma\in\mathbb{R}$.
For information regarding these classes we refer to \cite{duren-1983-book,goodman-1983-book,pommerenke-1975-book}.
There is a proper inclusion to hold among these classes.
\begin{align*}
\mathcal{C}\subsetneqq\mathcal{S}^{\ast}\subsetneqq\mathcal{K}\subsetneqq\mathcal{S}.
\end{align*}
Further a function $f\in\mathcal{A}_0$ is called pre-starlike function of order $\alpha$, $0\leq\alpha<1$ if
$f(z)\ast k_{\alpha}(z)\in\mathcal{S}^{\ast}(\alpha)$, \cite[p.48]{ruscheweyh-1982-book}.
This class is denoted by $\mathcal{R}^{\ast}(\alpha)$,
where $k_{\alpha}(z)=\frac{z}{(1-z)^{2-2\alpha}}$ plays the vital role
as it is the extremal function of $\mathcal{S}^{\ast}(\alpha)$ and for a complete account of details on $\mathcal{R}^{\ast}(\alpha)$ see
\cite{ruscheweyh-1982-book}. It is obvious that
$\mathcal{R}^{\ast}(1/2)\equiv\mathcal{S}^{\ast}(1/2)$ and $\mathcal{R}^{\ast}(0)\equiv\mathcal{C}$. Here the
Hadamard product or convolution denoted by $\ast$ is defined as follows:

Let $f(z)=\displaystyle\sum_{k=0}^{\infty}a_kz^k$ and $g(z)=\displaystyle\sum_{k=0}^{\infty}b_kz^k$, $z\in\mathbb{D}$. Then
\begin{align}\label{eqn:hadamard-product}
(f\ast g)(z):=\sum_{k=0}^{\infty}a_kb_kz^k, \quad \hbox{ for all $z\in\mathbb{D}$.}
\end{align}
In the present context, the following lemma is of considerable interest, which plays important role in several problems in function theory involving duality technique.
\begin{lemma}\cite[p. 54]{ruscheweyh-1977-prestarlike}
\label{lemma:ruscheweyh-prestarlike}
Let $F$ be prestarlike of order $0\leq \gamma<1$, $G\in\mathcal{S}^{\ast}(\gamma)$ and $H$ is any analytic
function in $\mathbb{D}$. Then,
\begin{align*}
\frac{F\ast (GH)}{F\ast G}(\mathbb{D})\subset \overline{co}(H(\mathbb{D})),
\end{align*}
where $co(A)$ is the convex hull of a set $A$.
\end{lemma}

 Another tool used in the sequel is the concept of subordination denoted by $\prec$.
An analytic function $f$ is subordinate to a univalent function $g$, written as
$f(z)\prec g(z)$, if there exists a Schwarz function $\omega(z):\mathbb{D}\rightarrow \mathbb{D}$, satisfying
$|\omega(z)|\leq |z|$ such that $f(z)=g(\omega(z))$.


To apply Theorem \ref{thm:cosine-positivity-sigma-bc} and Theorem \ref{thm:sine-positivity-sigma-bc}
in context of geometric function theory, we generalize the concept of
stable function by means of generalized Ces\`aro mean of type $(b-1;c)$.
For $f\in\mathcal{A}_1$ and $b+1>c>0$, the nth Ces\`aro mean of type ($b-1;c$) of
$f(z)=\displaystyle\sum_{k=0}^{\infty} a_kz^k \in\mathcal{A}_1$ is given by,
\begin{align}\label{eqn:definition-generalized-means}
\sigma_n^{(b-1,c)}(f,z):=\frac{1}{B_n}\sum_{k=0}^n B_{n-k} a_k z^k=\sigma_n^{(b-1,c)}(z)\ast f(z) ,
\quad \hbox{$n\in\mathbb{N}_0$},
\end{align}
where $B_k$ is defined in \eqref{eqn:definition-generalized-means}.
For $f\in\mathcal{A}$, we say $\sigma_n^{(b-1,c)}(f,z)$ is the
nth Ces\`aro mean of type $(b-1;c)$ of $f$. Geometric properties
of $\sigma_n^{(b-1,c)}(f,z)$ can be found in
\cite{sangal-swaminathan-geom-sigma-bc} and references therein.
Further $s_n(f,z)=\sigma_n^{(1-1,1)}(f,z)$ was studied by
Ruscheweyh with his collaborators, see \cite{ruscheweyh-salinas-2004-stable-JMAA}
and references therein. Similarly $\sigma_n^{\alpha}(f,z)=\sigma_n^{(1+\alpha-1,1)}(f,z)$
was studied by Saiful and Swaminathan in \cite{saiful-2012-stable-results-in-math}.

\section{ Generalized Ces\`aro stable function}
Using simple computation, \eqref{eqn:definition-generalized-means} can be rewritten in the following form:
\begin{align}\label{eqn:definition-Sn-in-n-1}
\sigma_n^{(b-1,c)}(f,z)=\left(\frac{c+n-1}{b+n-1}\right)\sigma_{n-1}^{(b-1,c)}(f,z)+
\frac{(b-c)}{b_n}\sum_{k=0}^{n-2}\frac{b_{n-k-1}}{(c+n-k-1)}a_kz^k\nonumber\\
+\frac{1}{b_n}\left(\frac{1+b-2c}{c}\right)a_{n-1}z^{n-1}+\frac{b_0}{b_n}a_nz^n.
\end{align}
In the sequel we denote $f_{\mu}(z):=\frac{1}{(1-z)^{\mu}}$ which
satisfies the following relations that are easy to verify.
\begin{align*}
\sigma_n^{(b-1,c)}(f_{\mu},z)'&=\left(\frac{c+n-1}{b+n-1}\right)\sigma_{n-1}^{(b-1,c)}(f_{\mu}',z), \\
z \sigma_n^{(b-1,c)}(f_{\mu},z)' &=\sigma_n^{(b-1,c)}(zf_{\mu}',z),\\
f_{\mu}-\frac{(1-z)}{\mu}f_{\mu}'&\equiv 0. 
\end{align*}
Now we state
the main result of this section. For the proof, we follow the procedure similar to the one given for Theorem 1.1 of \cite{ruscheweyh-salinas-2004-stable-JMAA}.

\begin{theorem}\label{thm:generalized-stable}
For $b\geq \max \{c,2c-1\}>0$ and $\mu\in[-1,1]$, the following equation holds.
\begin{align}\label{eqn:generalized-stable}
(1-z)^{\mu}\sigma_n^{(b-1,c)}(f_{\mu},z) \prec (1-z)^{\mu}.
\end{align}
\end{theorem}

\begin{proof}
The nth Ces\`aro mean of type $(b-1;c)$ of
$f(z)=\displaystyle\sum_{k=0}^{\infty}a_k z^k \in \mathcal{A}_1$ is given in \eqref{eqn:definition-generalized-means}.
Let $h(z):=1-(1-z)\sigma_n^{(b-1,c)}(f_{\mu},z)^{\frac{1}{\mu}}$.
In order to prove our result it is sufficient to prove $|h(z)|\leq 1$.
Clearly, for $\mu=0$, $f_{\mu}=1$ and hence $|h(z)|\leq 1$. We consider
the reminder of the proof in two parts based on the range of $\mu$. For the
first part, let $\mu\in(0,1]$. Consider
\begin{align}
(1-z)\sigma_n^{(b-1,c)}(f_{\mu},z)'&=\sigma_n^{(b-1,c)}(f_{\mu},z)'-z\sigma_n^{(b-1,c)}(f_{\mu},z)' \nonumber \\
&=\left(\frac{c+n-1}{b+n-1}\right)\sigma_{n-1}^{(b-1,c)}(f_{\mu}',z)-\sigma_n^{(b-1,c)}(zf_{\mu}',z) \label{eqn:11}
\end{align}
Using $\eqref{eqn:definition-Sn-in-n-1}$, $\sigma_n^{(b-1,c)}(zf_{\mu}',z)$ can be rewritten as,
\begin{align}
\sigma_n^{(b-1,c)}(zf_{\mu}',z)=\left(\frac{c+n-1}{b+n-1}\right)\sigma_{n-1}^{(b-1,c)}(zf_{\mu}',z)
+\frac{(b-c)}{b_n}\sum_{k=0}^{n-2}\frac{B_{n-k-1}}{(c+n-k-1)}ka_kz^k\nonumber\\
+\left(\frac{1+b-2c}{c}\right)\frac{(n-1)a_{n-1}}{B_n}z^{n-1}
+\frac{B_0}{B_n}na_nz^n.\label{eqn:22}
\end{align}
After substituting the value of $a_k=\frac{(\mu)_k}{k!}$, from \eqref{eqn:11} and \eqref{eqn:22} we obtain,
\begin{align*}
&(1-z)\sigma_n^{(b-1,c)}(f_{\mu},z)' \\
&=\left(\frac{c+n-1}{b+n-1}\right)\sigma_{n-1}^{(b-1,c)}((1-z)f_{\mu}',z)
-\frac{(b-c)}{B_n}\sum_{k=0}^{n-2}\frac{B_{n-k-1}}{(c+n-k-1)}\frac{k(\mu)_k}{k!}z^k\\
&\quad \quad -\left(\frac{1+b-2c}{c}\right)\frac{(n-1)(\mu)_{n-1}}{B_n (n-1)!}z^{n-1}
-\frac{B_0}{B_n}\frac{n(\mu)_n}{n! }z^n .
\end{align*}
Therefore,
\begin{align*}
& \sigma_n^{(b-1,c)}(f_{\mu},z)-\frac{(1-z)}{\mu}\sigma_n^{(b-1,c)}(f_{\mu},z)'   \\
&=\left(\frac{c+n-1}{b+n-1} \right)\sigma_{n-1}^{(b-1,c)}\left(f_{\mu}-\frac{1-z}{\mu}f_{\mu}',z\right)
+\frac{(b-c)}{B_n}\sum_{k=0}^{n-2}\frac{B_{n-k-1}}{(c+n-k-1)}\left( \frac{(\mu)_k}{k!}+\frac{k(\mu)_k}{\mu k!}\right)z^k\\
&\quad \quad +\left(\frac{1+b-2c}{c}\right)\left(\frac{(\mu)_{n-1}}{(n-1)!}+\frac{(n-1)(\mu)_{n-1}}{\mu (n-1)!}\right) \frac{z^{n-1}}{B_n}     +\left(\frac{(\mu)_n}{n!}+\frac{n(\mu)_n}{\mu n!} \right)\frac{B_0}{B_n}z^n \\
&=\frac{(b-c)}{B_n}\sum_{k=0}^{n-2}\frac{B_{n-k-1}}{(c+n-k-1)}\frac{(\mu+1)_k}{k!}z^k
+\left(\frac{1+b-2c}{c}\right)\frac{1}{B_n}\frac{(\mu+1)_{n-1}}{(n-1)!} z^{n-1}+\frac{(\mu+1)_n}{n!}\frac{B_0}{B_n}z^n
\end{align*}
Further,
\begin{align*}
h'(z) &=\left[\sigma_n^{(b-1,c)}(f_{\mu},z)\right]^{\frac{1}{\mu}}-
        \frac{(1-z)}{\mu}\left[\sigma_n^{(b-1,c)}(f_{\mu},z)\right]^{\frac{1}{\mu}-1}\cdot\left[\sigma_n^{(b-1,c)}(f_{\mu},z)\right]'\nonumber\\
&=\left[ \sigma_n^{(b-1,c)}(f_{\mu},z)\right]^{\frac{1}{\mu}-1}\left[ \sigma_n^{(b-1,c)}(f_{\mu},z)-
\frac{(1-z)}{\mu}\sigma_n^{(b-1,c)}(f_{\mu},z)'\right]\\
&=\left[\sigma_n^{(b-1,c)}(f_{\mu},z) \right]^{\frac{1}{\mu}-1}\times
\Biggl[\frac{(b-c)}{B_n}\sum_{k=0}^{n-2}\frac{B_{n-k-1}}{(c+n-k-1)}\frac{(\mu+1)_k}{k!}z^k\\
&\quad \quad \quad +\left(\frac{1+b-2c}{c}\right)\frac{1}{B_n}\frac{(\mu+1)_{n-1}}{(n-1)!}z^{n-1}
+\frac{(\mu+1)_n}{n!}\frac{B_0}{B_n}z^n \Biggr].
\end{align*}
Clearly, $f_{\mu}(z)=(1-z)^{-\mu}=1+\mu z+\frac{\mu(\mu+1)}{2!}z^2+\cdots+\frac{(\mu)_k}{k!}z^k+\cdots$.
Since $0<\mu\leq 1$, the Taylor coefficients of $f_{\mu}$ are positive. Thus,
\begin{align*}
\left| \sigma_n^{(b-1,c)}(f_{\mu},z)\right|
                                      \leq \sum_{k=0}^n \frac{B_{n-k}}{B_n}\frac{(\mu)_k}{k!}|z|^k
                                        =\sigma_n^{(b-1,c)}(f_{\mu},|z|)
\end{align*}
We obtained that the Taylor coefficients of $h'(z)$ are positive and from the definition of $h(z)$,
we have $h(0)=0$ and $h(1)=1$. Hence,
\begin{align*}
|h(z)|=\left|\int_0^z h'(t)dt\right|\leq \int_0^1 |h'(tz)|dt\leq \int_0^1 h'(t)dt=1 ,\quad \hbox{$z\in\mathbb{D}$}.
\end{align*}
Now for the second case $-1\leq \mu<0 $, the coefficients of
$(1-z)^{-\mu}=1+\mu z+\frac{\mu(\mu+1)}{2!}z^2+\cdots+\frac{(\mu)_k}{k!}z^k+\cdots$ are negative except 1
and  $\sigma_n^{(b-1,c)}(f_{\mu},z)=1+\displaystyle\sum_{k=1}^n \frac{B_{n-k}}{B_n}\frac{(\mu)_k}{k!}z^k=1-b(z)$,
 where $b(z)$ has positive Taylor series coefficients. Therefore,
\begin{align*}
\sigma_n^{(b-1,c)}(f_{\mu},z)^{\frac{1}{\mu}-1} =(1-b(z))^{\frac{1}{\mu}-1}
                                             =1+\sum_{k=1}^{\infty} \frac{(1-\frac{1}{\mu})_k}{k!}(b(z))^k.
\end{align*}
 This implies, $\sigma_n^{(b-1,c)}(f_{\mu},z)^{\frac{1}{\mu}-1}$ has non-negative Taylor series coefficients and
following the same steps as in part one, we obtain the result.
\end{proof}

If we substitute $b=c=1$ then Theorem \ref{thm:generalized-stable} reduces to the following corollary given in \cite{ruscheweyh-salinas-2004-stable-JMAA}.
\begin{corollary}\cite{ruscheweyh-salinas-2004-stable-JMAA}
Let $s_n(z,f)$ denote the nth partial sum of $f(z)$. Then for $n\in\mathbb{N}\cup\{0\}$ and for $\mu\in[-1,1]$,
\begin{align*}
(1-z)^{\mu} s_n(z,f_{\mu})\prec (1-z)^{\mu}.
\end{align*}
\end{corollary}

Important member of $\mathcal{S}^{\ast}(\lambda)$ are $zf_{2-2\lambda}=\frac{z}{(1-z)^{2-2\lambda}}$
that plays the role of extremal function while studying several properties such as growth,
distortion etc. Clearly, from Theorem \ref{thm:generalized-stable} for $\lambda\in[1/2,1)$, we get
\begin{align}\label{eqn:starlike-stable-lambda}
(1-z)^{2-2\lambda}\sigma_n^{(b-1,c)}\left(\frac{1}{(1-z)^{2-2\lambda}},z\right)
\prec (1-z)^{2-2\lambda}, \hbox{\quad for all $z\in \mathbb{D}$}.
\end{align}
It seems that starlike function of order $\lambda$, $\lambda\in[1/2,1)$ is
comparably a much narrow class but on the other side it has several interesting properties.
For example, our next theorem exhibits that \eqref{eqn:starlike-stable-lambda}
remains valid while in the left hand side of \eqref{eqn:starlike-stable-lambda},  $f_{2-2\lambda}$ is replaced by
any $f\in\mathcal{S}^{\ast}(\lambda)$ for $\lambda\in[1/2,1)$.
\begin{theorem}\label{thm:starlike-generalized-stable}
Let $f\in \mathcal{S}^{\ast}(\lambda)$, for $\lambda\in[\frac{1}{2},1)$, then
\begin{align}\label{eqn:starlike-generalized-stable}
\frac{z\sigma_n^{(b-1,c)}(f/z,z)}{f} \prec (1-z)^{2-2\lambda}, \hbox{ $\forall z\in \mathbb{D}$}.
\end{align}
\end{theorem}
\begin{proof}
Let $f\in\mathcal{S}^{\ast}(\lambda)$, then $\exists $ a unique prestarlike
function $F(z)$ of order $\lambda$ such that\\
 $f(z)=zf_{2-2\lambda}\ast F(z)$.
Then from Theorem \ref{thm:generalized-stable},
\begin{align*}
\frac{\sigma_n^{(b-1,c)}(f_{2-2\lambda},z)}{f_{2-2\lambda}}\prec \frac{1}{f_{2-2\lambda}}
\quad \hbox{for $\lambda\in[\frac{1}{2},1),z\in\mathbb{D}$}.
\end{align*}
Using Lemma \ref{lemma:ruscheweyh-prestarlike},
\begin{align*}
\frac{z \sigma_n^{(b-1,c)}(f/z,z)}{f} &=\frac{z(\sigma_n^{(b-1,c)}(z)*\frac{f(z)}{z})}{f}
                                        =\frac{z\sigma_n^{(b-1,c)}(z)*f(z)}{F(z)*zf_{2-2\lambda}}\\
                                        &=\frac{z\sigma_n^{(b-1,c)}(z)*(F(z)*zf_{2-2\lambda})}{F(z)*zf_{2-2\lambda}}
                                        =\frac{F(z)*(z\sigma_n^{(b-1,c)}(z)*zf_{2-2\lambda})}{F(z)*zf_{2-2\lambda}}\\
                                        &=\frac{F(z)*\left(zf_{2-2\lambda}.\frac{\sigma_n^{(b-1,c)}(f_{2-2\lambda},z)}{f_{2-2\lambda}} \right)}{F(z)*zf_{2-2\lambda}}
                                        \in \overline{co}\left( \frac{\sigma_n^{(b-1,c)}(f_{2-2\lambda},z)}{f_{2-2\lambda}} (\mathbb{D})\right),
\end{align*}
This means by Lemma \ref{lemma:ruscheweyh-prestarlike}, the range of $\frac{\sigma_n^{(b-1,c)}(f/z,z)}{f/z}$ lies in the closed convex hull of image of $\frac{\sigma_n^{(b-1,c)}(f_{2-2\lambda},z)}{f_{2-2\lambda}}$
under $\mathbb{D}$. From \eqref{eqn:starlike-stable-lambda}, for $\lambda\in[\frac{1}{2},1)$,
we have $\frac{\sigma_n^{(b-1,c)}(f_{2-2\lambda},z)}{f_{2-2\lambda}}\prec \frac{1}{f_{2-2\lambda}}$.
Therefore,
\begin{align*}
\frac{\sigma_n^{(b-1,c)}(f/z,z)}{f/z} \prec \frac{1}{f_{2-2\lambda}},
\end{align*}
which is equivalent to \eqref{eqn:starlike-generalized-stable} and the proof is
complete.
\end{proof}

Theorem \ref{thm:starlike-generalized-stable} has several consequences
with Kakeya Enestr$\ddot{o}$m theorem, that will be discussed in Section \ref{sec:consequences}.
Taking $b=c=1$, it reduces to the following result given by Ruscheweyh \cite{ruscheweyh-salinas-2004-stable-JMAA}.
\begin{corollary}\cite{ruscheweyh-salinas-2004-stable-JMAA}
Let $f\in\mathcal{S}^{\ast}(\lambda),\lambda\in[1/2,1)$. Then for $n\in\mathbb{N}\cup\{0\}$,
\begin{align*}
\frac{zs_n(z,f/z)}{f}\prec \frac{1}{f_{2-2\lambda}}.
\end{align*}
\end{corollary}

\begin{remark}\label{remark:cesaro-stable}
If we take $b=1+\beta$ and $c=1$, then it was proved in \cite{saiful-2012-stable-results-in-math} that for $\beta\geq 0$,
\begin{align*}
\frac{\sigma_n^{\beta}(f_{\mu},z)}{f_{\mu}}\prec \left\{
                                              \begin{array}{ll}
                                                \frac{1}{f_{\mu-\beta}}, & \hbox{$\mu\in[-1,0]$;} \\
                                                \frac{1}{f_{\mu+\beta}}, & \hbox{$\mu\in(0,1]$ such that $\mu+\beta\leq 1$.}
                                              \end{array}
                                            \right.
\end{align*}
The condition $\mu+\beta\leq1$ restricts $\beta$ to lie in $[0,1]$ where as Theorem \ref{thm:generalized-stable}
does not impose an upper bound on $\beta$ and moreover
\begin{align*}
\frac{1}{f_{\mu}} &\prec \frac{1}{f_{\mu+\beta}}, \quad \hbox{$\mu\in(0,1] $ where $\mu+\beta\leq1$}.\\
\frac{1}{f_{\mu}} &\prec \frac{1}{f_{\mu-\beta}}, \quad \hbox{$\mu\in[-1,0]$}.
\end{align*}
So, Theorem \ref{thm:generalized-stable} improves the result in \cite[Theorem 2.2]{saiful-2012-stable-results-in-math}.
A similar comparison can be made for
Theorem \ref{thm:starlike-generalized-stable} with \cite[Theorem 2.3]{saiful-2012-stable-results-in-math}.
\end{remark}

Theorem \ref{thm:starlike-generalized-stable} leads to a new
definition of generalized Ces\`aro stable functions.

\begin{definition}[Generalized Ces\`aro Stable Function]
\label{def:gen-stable}
A function $f\in\mathcal{A}_1$ is said to be n-generalized
Ces\`aro stable with respect to $F\in\mathcal{A}_1$ if
\begin{align}\label{eqn:definition-gen-cesaro-stable}
\frac{\sigma_n^{(b-1,c)}(f,z)}{f(z)}\prec \frac{1}{F(z)}
\end{align}
holds for some $n\in\mathbb{N}$. We call $f$ as n-generalized Ces\`aro stable
if it is n-generalized Ces\`aro stable with respect to itself.
If it is n-generalized Ces\`aro stable with respect to $F(z)$ for every n,
then it is said to be generalized Ces\`aro stable with respect
to $F(z)$.
\end{definition}

\begin{remark}
If we take $b=1+\beta, c=1$ then \eqref{eqn:definition-gen-cesaro-stable} reduces to
\begin{align*}
\frac{S_n^{\beta}(f,z)}{f} \prec \frac{1}{F(z)}
\end{align*}
gives the $(n,\beta)$ Ces\`aro-stability of $f$ about $F(z)$
\cite{saiful-2012-stable-results-in-math} which if
$\beta=0$ further reduces to stability of $f$ about $F(z)$
\cite{ruscheweyh-salinas-2004-stable-JMAA}.
\end{remark}

\begin{lemma}\cite[Proposition 5]{koumandos-ruscheweyh-2007-conjecture-JAT}\label{lemma:koumandos-ruscheweyh-2007-proposition}
For $\alpha,\beta>0$. If $F\prec(1-z)^{\alpha}$ and $G\prec(1-z)^{\beta}$ then $FG\prec(1-z)^{\alpha+\beta}$, $z\in\mathbb{D}$.
\end{lemma}

Now for $0<\mu\leq \rho \leq 1$, we have the following corollary of Theorem \ref{thm:generalized-stable}
following the same procedure as in \cite[page 57]{koumandos-ruscheweyh-2007-conjecture-JAT}.

\begin{corollary}
For $0<\mu\leq \rho \leq 1$, and  for $b\geq \max \{c,2c-1\}>0$ we have
\begin{align}\label{eqn:subord-mu-leq-rho}
(1-z)^{\rho}\sigma_n^{(b-1,c)}(f_{\mu},z) \prec (1-z)^{\rho}, z\in\mathbb{D}.
\end{align}
\end{corollary}
 The relation \eqref{eqn:subord-mu-leq-rho} is sharp in the sense that it will not hold for $\mu>\rho$. It is clear when
 $n$ becomes large then left hand side of \eqref{eqn:subord-mu-leq-rho} becomes unbounded and is subordinate to a
 bounded domain which is not possible.

If we change the right hand side of \eqref{eqn:subord-mu-leq-rho} by replacing the bounded function $(1-z)^{\rho}$, $0\leq\rho< 1$ by the unbounded one
$\left(\frac{1+z}{1-z}\right)^{\rho}$, $0\leq\rho< 1$, then the subordination
in \eqref{eqn:subord-mu-leq-rho} is still valid because
$(1-z)^{\rho}\prec\left(\frac{1+z}{1-z}\right)^{\rho}$ in $\mathbb{D}$.
Now this becomes a very interesting problem and leads to some new directions.
This situation leads to the following definition.

\begin{definition}
For $\rho\in(0,1]$, define $\mu(\rho,b-1,c)$ as the maximal number such that
\begin{align}\label{eqn:conj1}
(1-z)^{\rho} \sigma_n^{(b-1,c)}(f_{\mu},z)\prec \left(\frac{1+z}{1-z}\right)^{\rho},\quad n\in\mathbb{N}
\end{align}
holds for all $0<\mu\leq\mu(\rho,b-1,c)$.
\end{definition}
Writing
\begin{align*}
(1-z)^{2\rho-1}\sigma_n^{b-1,c}(f_{\mu},z)
=(1-z)^{\rho}\sigma_n^{b-1,c}(f_{\mu},z)\frac{1}{(1-z)^{1-\rho}}
\end{align*}
Then \eqref{eqn:conj1} implies,
\begin{align}\label{eqn:conj11}
\mathrm{Re} \left[(1-z)^{2\rho-1}\sigma_n^{b-1,c}(f_{\mu},z)\right]>0,
\quad \hbox{$z\in\mathbb{D}$ and $n\in\mathbb{N}$}.
\end{align}

Motivated by Conjecture 1 given in \cite{koumandos-ruscheweyh-2007-conjecture-JAT},
numerical evidences suggests the validity of the following conjecture
given below.

\begin{conj}\label{conj:1}
For $\rho\in(0,1]$ we have $\mu(\rho,b-1,c)=\mu^{\ast}(\rho,b-1,c)$,
where $\mu^{\ast}(\rho,b-1,c)$ is the unique solution in (0,1] of the equation
\begin{align}\label{eqn:conj1-integral}
\int_0^{(\rho+1)\pi}\sin(t-\rho\pi)t^{\mu-1}\left(1-\frac{t}{(\rho+1)\pi}\right)^{b-c}dt=0.
\end{align}
\end{conj}
 Conjecture \ref{conj:1} for the case $\rho=1/2$ is verified in
 Theorem \ref{thm:conj1-rho=1/2}, which justifies validity for the existence of conjecture \ref{conj:1}. Note that the case $\rho=3/4$ and $1/4$ with
 $b=1,c=1$ are addressed in \cite{koumandos-ruscheweyh-2007-conjecture-JAT,
 koumandos-lamprecht-2010-conjecture-JAT}. The authors have provided affirmative answer for the conjecture for several ranges including the one given
 in \cite{koumandos-lamprecht-2010-conjecture-JAT} in a separate work. Conjecture \ref{conj:1} contains the following weaker one.
\begin{conj}\label{conj:2}
Let $\rho\in(0,1]$ and $\mu^{\ast}(\rho,b-1,c)$ be as in Conjecture \ref{conj:1}, then
\begin{align}
\mathrm{Re}\left[(1-z)^{2\rho-1}\sigma_n^{(b-1,c)}(f_{\mu},z)\right]>0, \quad z\in\mathbb{D}, n\in \mathbb{N}
\end{align}
holds for $0<\mu\leq\mu^{\ast}(\rho,b-1,c)$ and $\mu^{\ast}(\rho,b-1,c)$ is the largest number with this
property.
\end{conj}
If we take $b=1+\beta$ and $c=1$ then $\sigma_n^{(b-1,;c)}(z)$ reduces to Ces\`aro mean of order $\beta$.
The following figures show graph of $\mu^{\ast}(\rho,\beta,1)$ for
$\beta=0,1,2,3.$
\begin{center}
\includegraphics[width=.4\columnwidth]{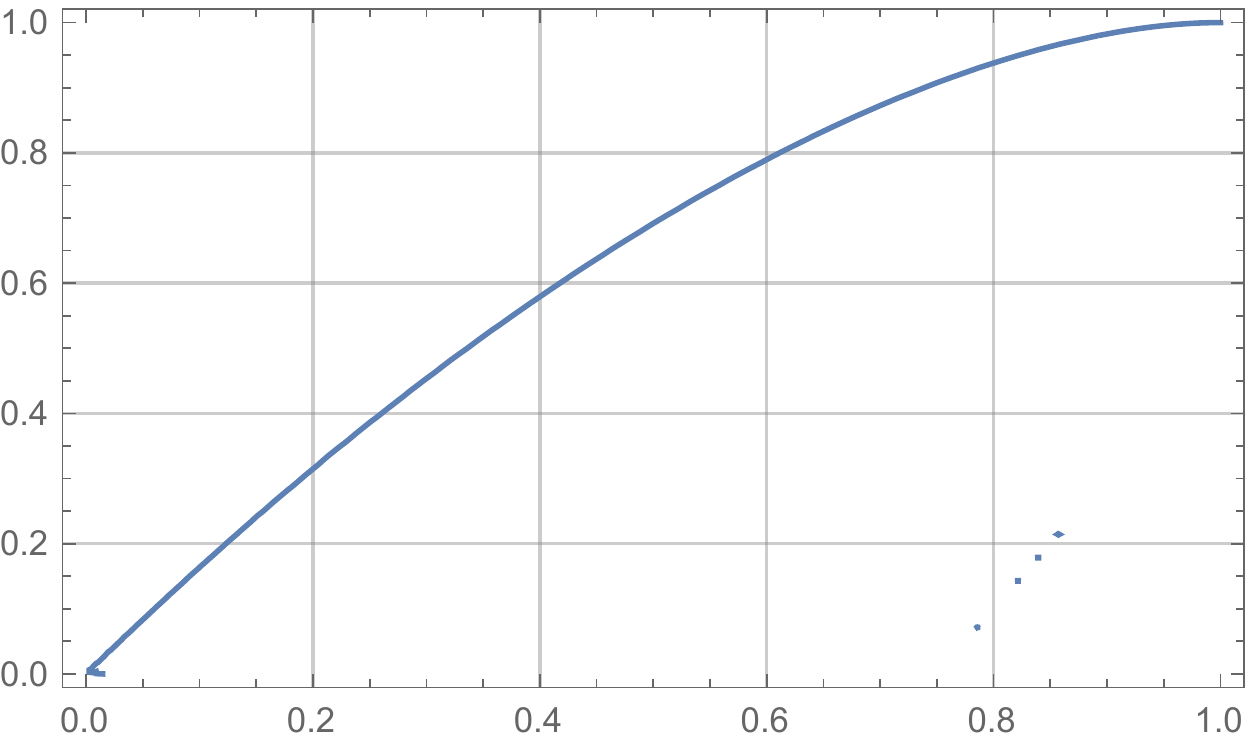}
\includegraphics[width=.4\columnwidth]{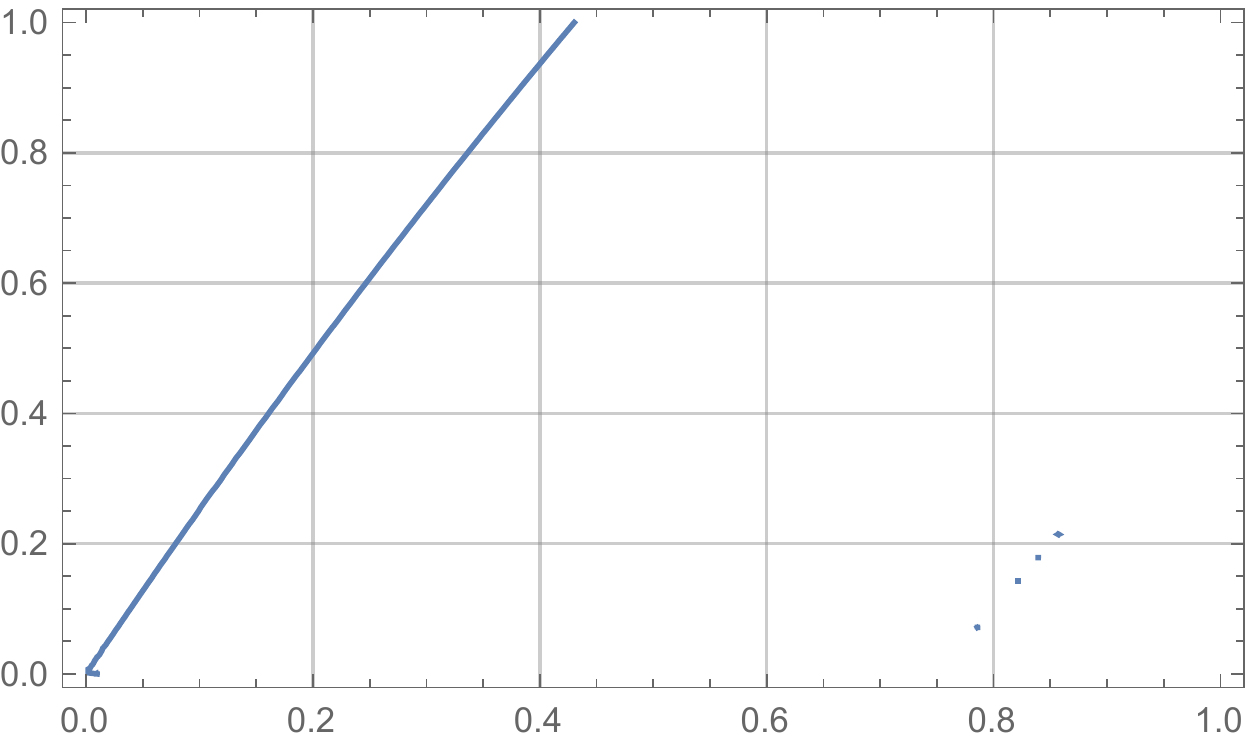}
\includegraphics[width=.4\columnwidth]{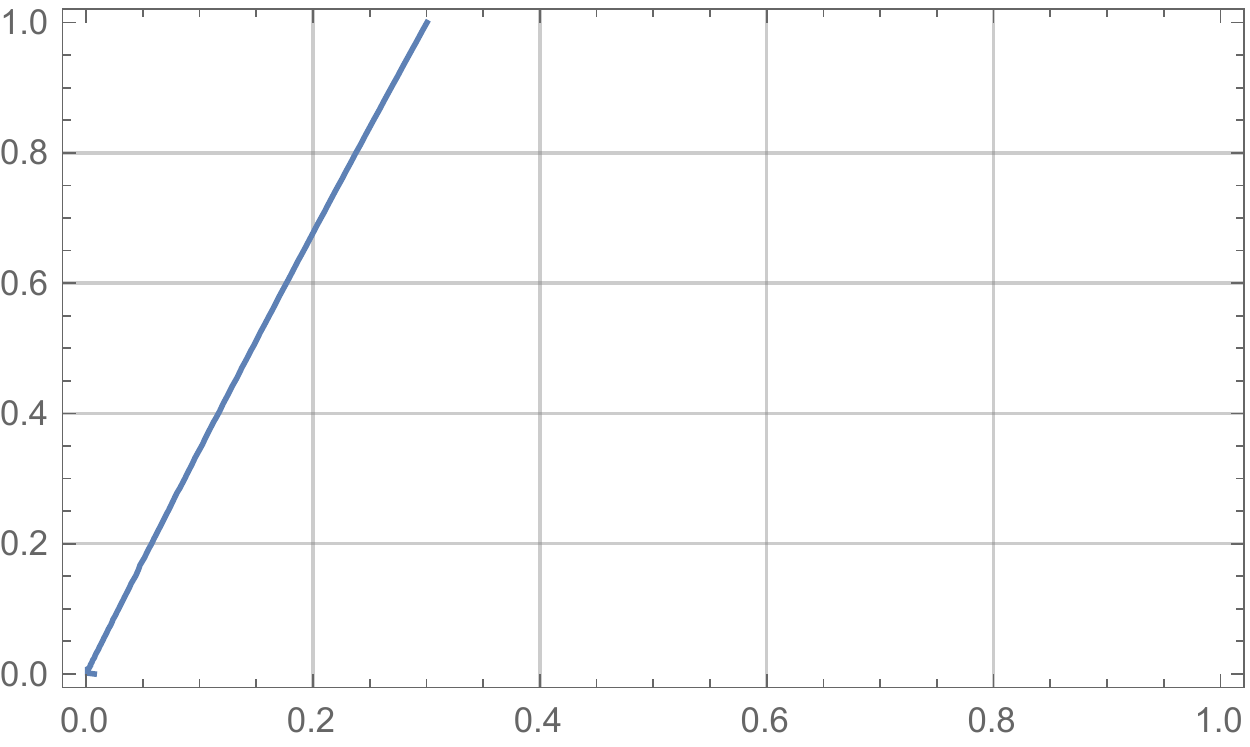}
\includegraphics[width=.4\columnwidth]{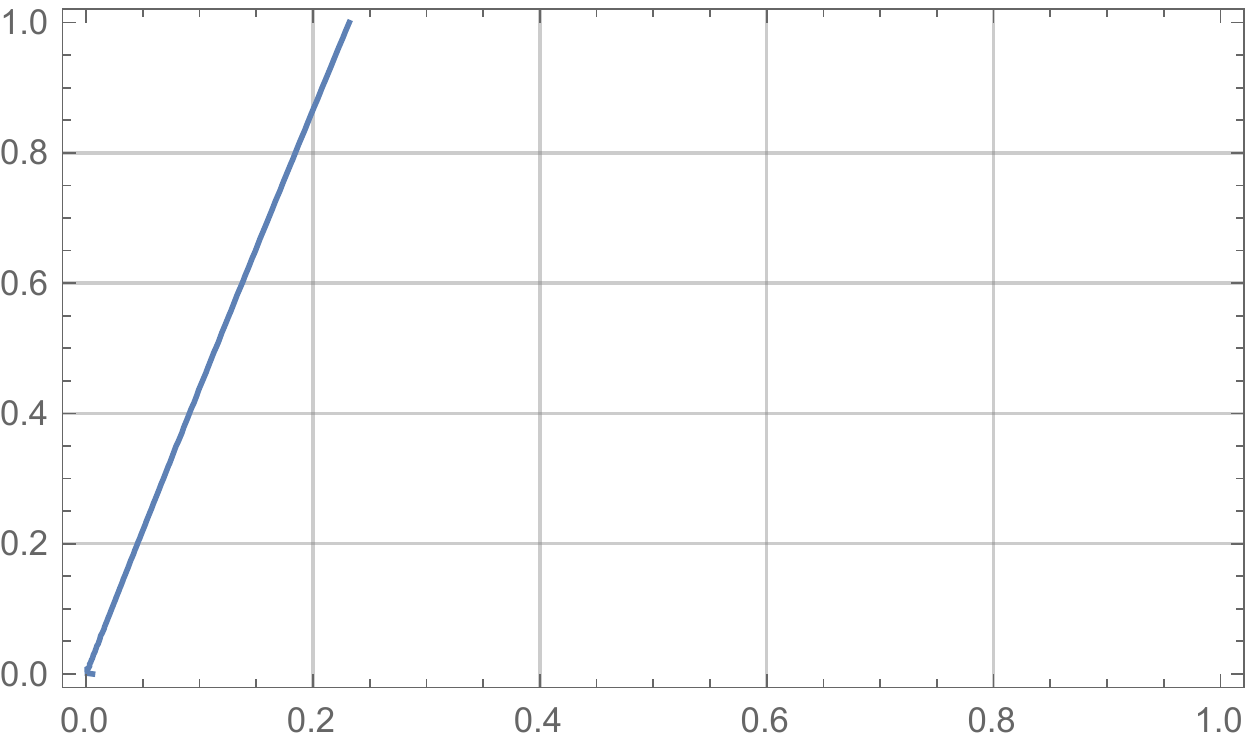}
\end{center}
If $\beta=0$, then the first figure is same as figure of $\mu^{\ast}$ given in
\cite{koumandos-ruscheweyh-2007-conjecture-JAT}.
For $\rho=1$, both conjectures are equivalent and reduces to
\begin{align*}
(1-z)\sigma_n^{(b-1;c)}(f_{\mu},z)\prec \left(\frac{1+z}{1-z}\right)
\end{align*}
which holds for $0<\mu\leq 1$.

For $\mu(\rho,b-1,c)$ and $\mu^{\ast}(\rho,b-1,c)$, we have the
following proposition.
\begin{proposition}
For $0<\rho<1$, we have $\mu^{\ast}(\rho,b-1,c)\geq \mu(\rho,b-1,c)$.
\end{proposition}
\begin{proof}
For $z=e^{i\phi}$, \eqref{eqn:conj11} is equivalent to
\begin{align}\label{eqn:prop-eqn1}
\sum_{k=0}^n \frac{B_{n-k}}{B_n}\frac{(\mu)_k}{k!}
\sin\left[(k+\rho-1/2)\phi-\rho\pi\right]<0,\quad \hbox{ for $0<\phi<2\pi$}
\end{align}
Now limiting case of this inequality can be obtained using the asymptotic formula,
\begin{align}\label{eqn:integral-asymptotic}
\lim_{n\rightarrow \infty} \left(\frac{\phi}{n}\right)^{\mu} \sum_{k=0}^n \frac{B_{n-k}}{B_n}\frac{(\mu)_k}{k!}
\sin\left[\left(k+\rho-1/2\right)\frac{\phi}{n}-\rho\pi\right]\nonumber\\
=\frac{1}{\Gamma(\mu)}\int_0^{\phi} t^{\mu-1}\left(1-\frac{t}{\phi}\right)^{b-c}
\sin(t-\rho\pi)dt
\end{align}
Hence a necessary condition for the validity of \eqref{eqn:prop-eqn1}
is the non positivity of the integral
\eqref{eqn:integral-asymptotic}. In particular, $\phi=(\rho+1)\pi$ gives
\begin{align*}
I^{b-1,c}(\mu)=\int_0^{(\rho+1)\pi} \sin(t-\rho\pi) t^{\mu-1}\left(1-\frac{t}{(\rho+1)\pi}\right)^{b-c}dt.
\end{align*}
We prove that $I^{b-1,c}(\mu)$ is strictly increasing function in $(0,1)$.
Now differentiation under integral sign gives
\begin{align*}
I^{b-1,c}(\mu)'&=\int_0^{(\rho+1)\pi}\sin(t-\rho\pi)\left(1-\frac{t}{(\rho+1)\pi}\right)^{b-c}t^{\mu-1}\log(1/t)dt\\
&= \left(1-\frac{t}{(\rho+1)\pi}\right)^{b-c}\int_0^{(\rho+1)\pi}\frac{\sin(t-\rho\pi)}{t^{1-\mu}}\log(1/t)dt \\
&\quad \quad + (b-c)\int_0^{(\rho+1)\pi}\left(1-\frac{t}{(\rho+1)\pi}\right)^{b-c-1}
\int_0^{(\rho+1)\pi}\frac{\sin(t-\rho\pi)}{t^{1-\mu}}\log(1/t)dt
\end{align*}
The positivity of $I^{b-1,c}(\mu)'$ follows from the increasing property of the integral $I(\mu)$ in
\cite[Lemma 1]{koumandos-ruscheweyh-2007-conjecture-JAT} using the method of
Zygmund \cite[V. 2.29]{zygmund-2002-book-trig-series}.
So $I^{b-1,c}(\mu)$ is strictly increasing in $(0,1)$ and if we choose $b=c$ then $I(0)=-\infty$
and $I(1)>0$, so $I(\mu)=0$ has only one solution in $(0,1]$ which is $\mu^{\ast}(\rho,b-1,c)$ given by \eqref{eqn:conj1-integral}.
Hence the best possible bound for $\mu$ in Conjecture \ref{conj:2} cannot be greater than
$\mu^{\ast}(\rho,b-1,c)$. This proves the assertion.
\end{proof}

Since the conditions in Conjecture \ref{conj:1} and Conjecture \ref{conj:2} turns out to be the
positivity of trigonometric polynomials. So it follows from summation by parts
that both conjectures need to established only for
$\mu=\mu^{\ast}(\rho,b-1,c)$. We discuss some particular cases of these conjectures.

\begin{theorem}\label{thm:conj1-rho=1/2}
Conjecture \ref{conj:1} holds for $\rho=1/2$.
\end{theorem}
\begin{proof}
If $\rho=1/2$ then \eqref{eqn:conj1} is equivalent to
\begin{align}\label{eqn:conj1-1/2}
\mathrm{Re}[(1-z)\sigma_n^{b-1,c}(f_{\mu},z)^2]>0
\end{align}
Using minimum principle for harmonic functions it is sufficient to establish
\eqref{eqn:conj1-1/2} for $z=e^{2i\phi}, 0<\phi<\pi$.
Let
\begin{align}
P_n(\phi):=(1-e^{2i\phi})\left\{\sum_{k=0}^n \frac{B_{n-k}}{B_n}\frac{(\mu)_k}{k!}e^{2ik\phi}\right\}^2
\end{align}
and we want to prove $\mathrm{Re}P_n(\phi)>0$ for all $n\in\mathbb{N}$, $0<\phi<\pi.$

For arbitrary number $d_k=c_{2k}=c_{2k+1}$, $k=0,1,2,\ldots,n$, we have
\begin{align*}
(1+z)\sum_{k=0}^nd_kz^{2k}&=\sum_{k=0}^{2n+1}c_k z^k,\\
\hbox{ and }  (1-z)\sum_{k=0}^nd_kz^{2k}&=\sum_{k=0}^{2n+1}(-1)^k c_k z^k,\\
\hbox{ so that } (1-z^2)\left[\sum_{k=0}^nd_kz^{2k}\right]^2&=
\left(\sum_{k=0}^{2n+1}c_k z^k\right)\left(\sum_{k=0}^{2n+1}(-1)^k c_k z^k\right).
\end{align*}
Choosing $z=e^{i\phi}, -z=e^{-i(\pi-\phi)}$ we have
\begin{align*}
(1-e^{2i\phi})\left(\sum_{k=0}^n d_k e^{2ik\phi}\right)^2=
\left(\sum_{k=0}^{2n+1}c_k e^{ik\phi} \right)\left(\sum_{k=0}^{2n+1}(-1)^k c_k e^{ik\phi}\right),
\end{align*}
which implies
\begin{align*}
&\mathrm{Re}(P_n(\phi))\\
&=\left(\sum_{k=0}^{2n+1}c_k\cos{k\phi}\right)
\left(\sum_{k=0}^{2n+1}c_k\cos{k(\pi-\phi)}\right)+\left(\sum_{k=1}^{2n+1}c_k\sin{k\phi}\right)
\left(\sum_{k=1}^{2n+1}c_k\sin{k(\pi-\phi)}\right).
\end{align*}
Since $c_{2k}=c_{2k+1}$, we have
\begin{align*}
\sin{\frac{\phi}{2}}\sum_{k=0}^{2n+1}c_k\cos{k\phi}
=\cos{\frac{\phi}{2}}\sum_{k=1}^{2n+1}c_k\sin{k(\pi-\phi)}.
\end{align*}
This leads to the fact that
\begin{align}\label{eqn:odd-cos-sum}
\sum_{k=0}^{2n+1} c_k \cos k\phi>0, \quad 0<\phi<\pi
\end{align}
and
\begin{align}\label{eqn:odd-sine-sum}
\sum_{k=1}^{2n+1} c_k \sin k\phi>0, \quad 0<\phi<\pi
\end{align}
are equivalent. When $d_k=\frac{B_{n-k}}{B_n}\frac{(\mu)_k}{k!}$ then positivity of \eqref{eqn:odd-cos-sum}
and \eqref{eqn:odd-sine-sum} hold respectively from Theorem \ref{thm:cosine-positivity-sigma-bc} and
Theorem \ref{thm:sine-positivity-sigma-bc} for $0<\mu\leq \mu_0'$ and $0<\phi<\pi$.
So $\mathrm{Re}(P_n(\phi))>0$ which means Conjecture \ref{conj:1} is true for $\rho=1/2$.
\end{proof}

As we have seen that Theorem \ref{thm:conj1-rho=1/2} becomes equivalent to the
extension of Vietori's theorem \cite{sangal-swaminathan-positivity-alpha-beta} an
interpretation of extension of Vietori's theorem in terms of
generalized Ces\`aro stable functions is obtained in section 2.


For further generalization of Theorem \ref{thm:starlike-generalized-stable}, we define for $\mu>0$,
\begin{align*}
\mathcal{F}_{\mu}:=\left\{f\in \mathcal{A}_0: \mathrm Re\left(\frac{zf'}{f}\right)>\frac{-\mu}{2}
,z\in\mathbb{D}\right\},
\end{align*}
and $f_{\mu}=\frac{1}{(1-z)^{\mu}}$ taken as an extremal function for $\mathcal{F}_{\mu}$. For all
$f\in\mathcal{F}_{\mu}$ we get $f\prec f_{\mu}$.
It is obvious that $ f\in\mathcal{F}_{\mu} \Leftrightarrow zf\in\mathcal{S}^{\ast}(1-\mu/2)$.
We define
\begin{align*}
\mathcal{PF}_{\mu}=\{f\in \mathcal{A}_0: f\ast f_{\mu}\in \mathcal{F}_{\mu}\}.
\end{align*}
Clearly $\mathcal{PF}_1=\mathcal{F}_1$. The functions of $\mathcal{F}$ and $\mathcal{PF}$ behaves same as the
functions of starlike and prestarlike classes respectively.
Before going to proceed further we recall some results on starlike and prestarlike class.

\begin{lemma}\cite{ruscheweyh-1977-prestarlike}
For $0<\mu\leq \rho$, we have
\begin{enumerate}
\item $\mathcal{F}_{\mu} \subset \mathcal{F}_{\rho}$
\item $\mathcal{PF}_{\mu} \supset \mathcal{PF}_{\rho}$
\item If $h\in\mathcal{PF}_{\mu}$ and $f\in\mathcal{F}_{\mu}$ then $h\ast f \in\mathcal{F}_{\mu}$.
\end{enumerate}
\end{lemma}

Lemma \ref{lemma:ruscheweyh-prestarlike} also holds good in context with the class $\mathcal{F}_{\mu}$ and $\mathcal{PF}_{\mu}$. We need the following lemma.

We define $\tilde{f}_{\mu}\in\mathcal{A}_0$ be the unique solution of $f_{\mu}\ast\tilde{f}_{\mu}=\frac{1}{1-z}$.
It is clear that $f\in\mathcal{F}_{\mu} \Longleftrightarrow f\ast \tilde{f}_{\mu}\in  \mathcal{PF}_{\mu}$.

\begin{theorem}\label{thm:extension-starlike-stable-hypergeometric}
Let $\rho\in(0,1]$ and $f\in\mathcal{S}^{\ast}(1-\mu/2)$ with $0<\mu\leq \rho$, then for
$b\geq\max\{c,2c-1\}>0$,
\begin{align}\label{eqn:extension-starlike-stable-hypergeomertic}
\frac{\sigma_n^{(b-1,c)}(f,z)}{\phi_{\rho,\mu}\ast f} \prec (1-z)^{\rho},\quad \hbox{ $n\in\mathbb{N}$},
\end{align}
where $\phi_{\rho,\mu}(z)=zF(1,\rho;\mu;z)$ , where $F$ is the Gaussian hypergeometric function can
also be defined by the equation,
\begin{align*}
\frac{z}{(1-z)^{\mu}}\ast \phi_{\rho,\mu}=\frac{z}{(1-z)^{\rho}}.
\end{align*}
\end{theorem}

\begin{proof}
Let $\phi_{\rho,\mu}(z)=\displaystyle\sum_{k=0}^{\infty}\frac{(\rho)_k}{(\mu)k}z^k=f_{\rho}\ast \tilde{f}_{\mu}$
where $\tilde{f}_{\mu}$ is defined as $f_{\mu}\ast\tilde{f}_{\mu}=\frac{1}{1-z}$.
For $0<\mu<\rho\leq 1$, $f_{\rho-\mu}=\frac{1}{(1-z)^{\rho-\mu}}$ maps $\mathbb{D}$ univalently into a convex domain. $f\in\mathcal{F}_{\mu} \Rightarrow f\ast\tilde{f}_{\mu}\in \mathcal{PF}_{\mu}$ and $f_{\mu}\in\mathcal{F}_{\mu}$.
Clearly,
\begin{align*}
\frac{\phi_{\rho,\mu}\ast f}{f}
&=\frac{f_{\rho}\ast \tilde{f}_{\mu}\ast f}{f_{\mu}\ast\tilde{f}_{\mu}\ast f}
=\frac{f\ast \tilde{f}_{\mu}\ast f_{\mu}f_{\rho-\mu}}{f\ast\tilde{f}_{\mu}\ast f_{\mu}}
\in co\left(f_{\rho-\mu}(\mathbb{D}) \right),
\end{align*}
i.e. $\frac{\phi_{\rho,\mu}\ast f}{f}  \prec \frac{1}{(1-z)^{\rho-\mu}}$.
Since $f \in \mathcal{F}_{\mu}  \Rightarrow \frac{\sigma_n^{(b-1,c)}(f,z)}{f}\prec (1-z)^{\mu}$.
So using Lemma \ref{lemma:koumandos-ruscheweyh-2007-proposition},
\begin{align*}
\frac{\sigma_n^{(b-1,c)}(f,z)}{\phi_{\rho,\mu}\ast f}\prec (1-z)^{\rho}.
\end{align*}
If we take $zf\in\mathcal{S}^{\ast}(1-\mu/2)$ we get that,
\begin{equation*}
\frac{\sigma_n^{(b-1,c)}(f,z)}{\phi_{\rho,\mu}\ast f}\prec (1-z)^{\rho}.\mbox{ \qedhere}
\end{equation*}
\end{proof}

\begin{remark}
If we take $\rho=\mu=2-2\lambda$, then \eqref{eqn:extension-starlike-stable-hypergeomertic}
becomes \eqref{eqn:starlike-generalized-stable}. This means Theorem
\ref{thm:extension-starlike-stable-hypergeometric} can be regarded as a
generalization of Theorem \ref{thm:starlike-generalized-stable}.
\end{remark}

\section{Matrix Representation}
Ces\`aro mean of type $(b-1;c)$ can be written in terms of lower triangular matrix $(g_{ij})$
defined as,
\begin{align*}
g_{i0}=1, \quad \quad
g_{ik}=\left\{
         \begin{array}{ll}
           \frac{B_{i-k}}{B_i}, & \hbox{$1\leq k\leq i$;} \\
           0, & \hbox{$k\geq i+1$.}
         \end{array}
       \right.
\end{align*}
Then the entries in $(n+1)th$ row of the matrix induces Ces\`aro mean of type $(b-1;c)$ of
order $n$ is given by,
\begin{align*}
\sigma_n^{(b-1,c)}(z)=\sum_{k=0}^{n}\frac{B_{n-k}}{B_n}z^k, \quad z\in\mathbb{D}.
\end{align*}
Consider,
\begin{align*}
        G=\left(
                          \begin{array}{cccccc}
                            1 & 0 & 0 & 0 & \cdots & 0 \\
                            1 & \frac{B_0}{B_1} & 0 & 0 & \cdots & 0 \\
                            1 & \frac{B_1}{B_2} & \frac{B_0}{B_2} & 0 & \cdots & 0 \\
                            1 & \frac{B_2}{B_3} & \frac{B_1}{B_3} & \frac{b_0}{b_3} & \cdots & 0 \\
                           \vdots & \vdots & \vdots & \vdots & \cdots & \vdots \\
                          \end{array}
                        \right)
\times
\left(
\begin{array}{c}
a_0\\
a_1z\\
a_2z^2\\
a_3z^3\\
\vdots
\end{array}
\right)
\end{align*}
Then $(n+1)th$ row of G generates the Ces\`aro mean of type $(b-1,c)$ of
$f(z)=\displaystyle\sum_{k=0}^{\infty}a_kz^k$ of order $n$ for $n\geq0$.
Then the concept of stable function can be generalized in terms of lower triangular matrix as well.

For $n\in\mathbb{N}$, $\mathcal{H}_n$ be the set of lower triangular matrix $(h_{ij})$ of order $(n+1)$
satisfying $h_{ij}\geq0,i,j=0,1,2\ldots,n$, and satisfy the following conditions:
\begin{enumerate}
\item $h_{i0}=1$ for every $i=0,1,\ldots,n$,
\item for each fixed $i\geq1$, $h_{ij}=h_{i1}h_{i-1,j-1}$, $j=1,\ldots,n$,
\item for each fixed $i\geq1$, $\{h_{ij}\}$ is a decreasing sequence.
\end{enumerate}
Then $(n+1)th$ row of $H_n$ induces a polynomial $H_n$ of degree n is
\begin{align*}
H_n(z):=\sum_{k=0}^n h_{nk}z^k,
\end{align*}
and for $f(z)=\displaystyle \sum_{k=0}^{\infty}a_kz^k\in \mathcal{A}_1$ the polynomial
\begin{align}\label{eqn:poly-hn}
H_n(f,z)=\sum_{k=0}^n h_{nk}a_kz^k=H_n(z)\ast f(z).
\end{align}
Following the same procedure as in Theorem \ref{thm:generalized-stable} we can obtain the following theorem
for $H_n$ defined by lower triangular matrix. We state the result without proof.
\begin{theorem}
Let $H_n$ be given by \eqref{eqn:poly-hn}, and $f_{\mu}=1/(1-z)^{\mu}$. Suppose
$h_{n1}\leq 1$, then for $\mu\in[-1,1]$,
\begin{align}
(1-z)^{\mu} H_n(f_{\mu},z)\prec (1-z)^{\mu}.
\end{align}
\end{theorem}

\section{Application in geometric properties of Ces\`aro mean of type $(b-1;c)$}
For finding the geometric properties of Ces\`aro mean of type $(b-1;c)$,
instead of $\sigma_n^{b-1;c}(z)$ we will use
normalized Ces\`aro mean of type $(b-1;c)$ denoted by $s_n^{b-1;c}(z)$ because the geometric properties
like convexity, starlikeness and close-to-convexity remains intact under such normalization.
For $b+1>c>0$, let
\begin{align*}
s_n^{(b-1,c)}(z):=z+\sum_{k=2}^n \frac{B_{n-k}}{B_{n}}z^k, \quad z\in\mathbb{D}.
\end{align*}
For $f\in\mathcal{A}$, it is easy to obtain that
\begin{align*}
\left(\frac{b+n-1}{c+n-1}\right)s_n^{(b-1,c)}(f,z)'=\sigma_{n-1}^{(b-1,c)}(f',z)
=\sigma_{n-1}^{(b-1,c)}(z)\ast f'(z).
\end{align*}

Note that $s_n^{\beta;1}=s_n^{\beta}(z)$ was defined in \cite{ruscheweyh-1992-geom-cesaro}.
Among the results available in the literature regarding $s_n^{\beta}(z)$, the
interesting result is given by Lewis \cite{lewis-1979-closetoconvexity-cesaro-mean-SIAM}
is that for $\beta\geq 1$ and $n\in\mathbb{N}$, $s_n^{\beta}(z)\in\mathcal{K}$.
Using the convolution between convex and close-to-convex functions, it is clear that for
$f\in\mathcal{C}$, $(n+\beta)s_n^{\beta}(f,z)/n\in \mathcal{K}, \beta\geq1$.
Ruscheweyh and Salinas \cite{ruscheweyh-salinas-1993-subordination-cesaro} also discussed
the geometric property of $(n+\beta)s_n^{\beta}(f,z)/n$ when $0<\beta<1$.
It is interesting to discuss the geometric property of Ces\`aro mean of
type $(b-1,c)$ of $f(z)$, where $f(z)$ belongs to some class of functions.
Note that certain geometric properties of $s_n^{b-1;c}(z)$ are given in
\cite{sangal-swaminathan-geom-sigma-bc}, mainly using the positivity results that
are consequences of \cite{sangal-swaminathan-geom-sigma-bc}.
In this section, we provide some more geometric properties as consequences
of Theorem \ref{thm:generalized-stable} and Theorem \ref{thm:starlike-generalized-stable}
which are fundamental in the formulation of Definition \ref{def:gen-stable}.

\begin{theorem}
Let $F_{\lambda}(z)=z+\displaystyle\sum_{k=2}^{\infty}(2-2\lambda)_{k-1}\frac{z^k}{k!}$,
$\lambda\in[1/2,1)$. Then for $b\geq\max\{c,2c-1\}>0$,
\begin{align*}
\left|1-(1-z)\cdot\left(
\left(\frac{b+n-1}{c+n-1}\right)(s_n^{(b-1,c)}(F_{\lambda},z))'\right)^{\frac{1}{2-2\lambda}}\right|\leq1.
\end{align*}
In particular, $\left(\frac{b+n-1}{c+n-1}\right)s_n^{(b-1,c)}(F_{\lambda},z)\in\mathcal{K}(\lambda)$.

\begin{proof}
It is given that,
\begin{align*}
F_{\lambda}(z)=z+\sum_{k=2}^{\infty} (2-2\lambda)_{k-1}\frac{z^k}{k!}, \quad \hbox{$\lambda\in[1/2,1)$}.
\end{align*}
By Alexander transform it is obvious that,
\begin{align}\label{eqn:stable-app-gft-alex-trans}
F_{\lambda}(z)\in \mathcal{C}(\lambda)& \Longleftrightarrow z F'_{\lambda}=\frac{z}{(1-z)^{2-2\lambda}}\in \mathcal{S}^{\ast}(\lambda).
\end{align}
Substituting $2-2\lambda=\mu$, we obtain
\begin{align*}
(1-z)^{2-2\lambda}\sigma_{n-1}^{(b-1,c)}\left(z,\frac{1}{(1-z)^{2-2\lambda}}\right)&\prec (1-z)^{2-2\lambda}.
\end{align*}
Since
\begin{align*}
(1-z)^{2-2\lambda}\sigma_{n-1}^{(b-1,c)}\left(z,\frac{1}{(1-z)^{2-2\lambda}}\right)
&=(1-z)^{2-2\lambda}\sigma_{n-1}^{(b-1,c)}(z,F'_{\lambda})\\
&=(1-z)^{2-2\lambda} \left(\frac{b+n-1}{c+n-1}\right)s_n^{(b-1,c)}(z,F_{\lambda})',
\end{align*}
we get, using Theorem \ref{thm:generalized-stable},
\begin{align*}
\left|1-\left((1-z)^{2-2\lambda}\cdot\left(\left(\frac{b+n-1}{c+n-1}\right)
(s_n^{(b-1,c)}(z,F_{\lambda}))'\right)\right)^{\frac{1}{2-2\lambda}}\right| &\leq1,
\end{align*}
which is equivalent to,
\begin{align*}
 \mathrm{Re} \bigg( (1-z)^{2-2\lambda}\cdot\left(\frac{b+n-1}{c+n-1}\right)\left(s_n^{(b-1,c)}(z,F_{\lambda})\right)'\bigg)>0.
\end{align*}
This expressions together with \eqref{eqn:stable-app-gft-alex-trans} and the analytic
characterization of $\mathcal{K}(\lambda)$ guarantees that
$\left(\frac{b+n-1}{c+n-1}\right)s_n^{(b-1,c)}(z,F_{\lambda})\in\mathcal{K}(\lambda)$
with respect to the starlike function given in \eqref{eqn:stable-app-gft-alex-trans}.
\end{proof}
\end{theorem}
In particular if $\lambda=1/2$, $F_{1/2}(z)=-\log(1-z)$, then
\begin{align*}
\left(\frac{b+n-1}{c+n-1}\right)s_n^{b-1,c}(-\log(1-z),z)\in \mathcal{K}(1/2).
\end{align*}

\begin{theorem}\label{thm:close-to-convexity-lambda}
If $f\in\mathcal{C}(\lambda)$, $\lambda\in[1/2,1)$ and $b\geq \max\{c,2c-1\}$, then for $n\geq1$,
\begin{align*}
\dfrac{\left(\frac{b+n-1}{c+n-1}\right)(s_n^{(b-1,c)}(f,z))'}{f'(z)}\prec (1-z)^{2-2\lambda}.
\end{align*}
In particular, $\left(\frac{b+n-1}{c+n-1}\right)s_n^{(b-1,c)}(z,f)\in\mathcal{K}(\lambda)$.
\begin{proof}
If $f\in\mathcal{C}(\lambda)$, then by Alexander transform, $g=zf'(z)\in\mathcal{S}^{\ast}(\lambda)$, then
\begin{align*}
\dfrac{\left(\frac{b+n-1}{c+n-1}\right)(s_n^{(b-1,c)}(z,f))'}{f'(z)}
=\frac{z \sigma_{n-1}^{(b-1,c)}(z,f')}{g}
=\frac{z \sigma_{n-1}^{(b-1,c)}(z,g/z)}{g}.
\end{align*}
If $g\in\mathcal{S}^{\ast}(\lambda)$, $\lambda\in[1/2,1)$, then from Theorem \ref{thm:starlike-generalized-stable},
\begin{align*}
\dfrac{\left(\frac{b+n-1}{c+n-1}\right)(s_n^{(b-1,c)}(z,f))'}{f'(z)}\prec (1-z)^{2-2\lambda}
\Rightarrow\mathrm{Re} \left(\dfrac{z \left(\frac{b+n-1}{c+n-1}\right)s_n^{(b-1,c)}(z,f)'}{g(z)}\right)>0
\end{align*}
This means
$\left(\frac{b+n-1}{c+n-1}\right)s_n^{(b-1,c)}(z,f) \in \mathcal{K}(\lambda)$. \qedhere
\end{proof}
\end{theorem}

If we substitute $b=1+\beta$ and $c=1$ in Theorem \ref{thm:close-to-convexity-lambda} then we obtain the following corollary.
\begin{corollary}
If $f\in\mathcal{C}(\lambda)$, $\lambda\in[1/2,1)$ and $\beta\geq0$ then
for $n\geq1$, $\frac{n+\beta}{n}s_n^{\beta}(f,z)\in\mathcal{K}(\lambda)$.
\end{corollary}


If we choose $g(z)=z$, for $f\in\mathcal{C}(\lambda)$ where $\lambda\in[1/2,1)$ then,
\begin{align*}
\mathrm{Re}(s_n^{(b-1,c)}(f,z))'>0 \Longrightarrow s_n^{(b-1,c)}(f,z)'\neq0.
\end{align*}
Since every close-to-convex function is univalent \cite[p.47]{duren-1983-book},
the generalized Ces\`aro mean $s_n^{(b-1,c)}(f,z)$ for the convex
function $f$ is also univalent. In this situation for $b=1,c=1$,
a subordination chain was provided by Ruscheweyh and Salinas
\cite{ruscheweyh-salinas-1993-subordination-cesaro} which is given
in the following result.
\begin{theorem}\cite{ruscheweyh-salinas-1993-subordination-cesaro}
If $f\in \mathcal{C}(1/2)$, then
\begin{align*}
s_1^{(\alpha+k)}(z,f)\prec s_2^{(\alpha+k)}(z,f)\prec \cdots s_n^{(\alpha+k)}(z,f)\prec\cdots f(z), \quad k\in\mathbb{N}.
\end{align*}
holds for $\alpha\geq0$ and $z\in\mathbb{D}$.
\end{theorem}
An extension of this result to $\sigma_n^{(b-1;c)}(f,z)$ can provide more information
on the geometric nature of $\sigma_n^{(b-1;c)}(z)$ and we state this as a problem.

\textbf{Open Problem.}
For $b\geq \max\{c,2c-1\}>0$ and
$f\in\mathcal{C}(\lambda)$ where $\lambda\in[1/2,1)$
we have the following subordination chain.
\begin{align}\label{eqn:open-prob-subordination-chain}
s_1^{(b-1+k,c)}(z,f)\prec s_2^{(b-1+k,c)}(z,f)\prec \cdots s_n^{(b-1+k,c)}(z,f)\prec\cdots f(z), \quad k\in\mathbb{N}.
\end{align}

We do not have the proof of this problem but
the graphical justification of the problem is provided here.
If we take $f(z)=-\log(1-z)=\displaystyle\sum_{k=1}^{\infty}\frac{z^k}{k}\in \mathcal{C}(1/2)$.
Then we have the following two graphs, first one is for n=1,2,3,4 and second is for n=4,5,6,7.
\begin{center}
\includegraphics[width=.4\columnwidth]{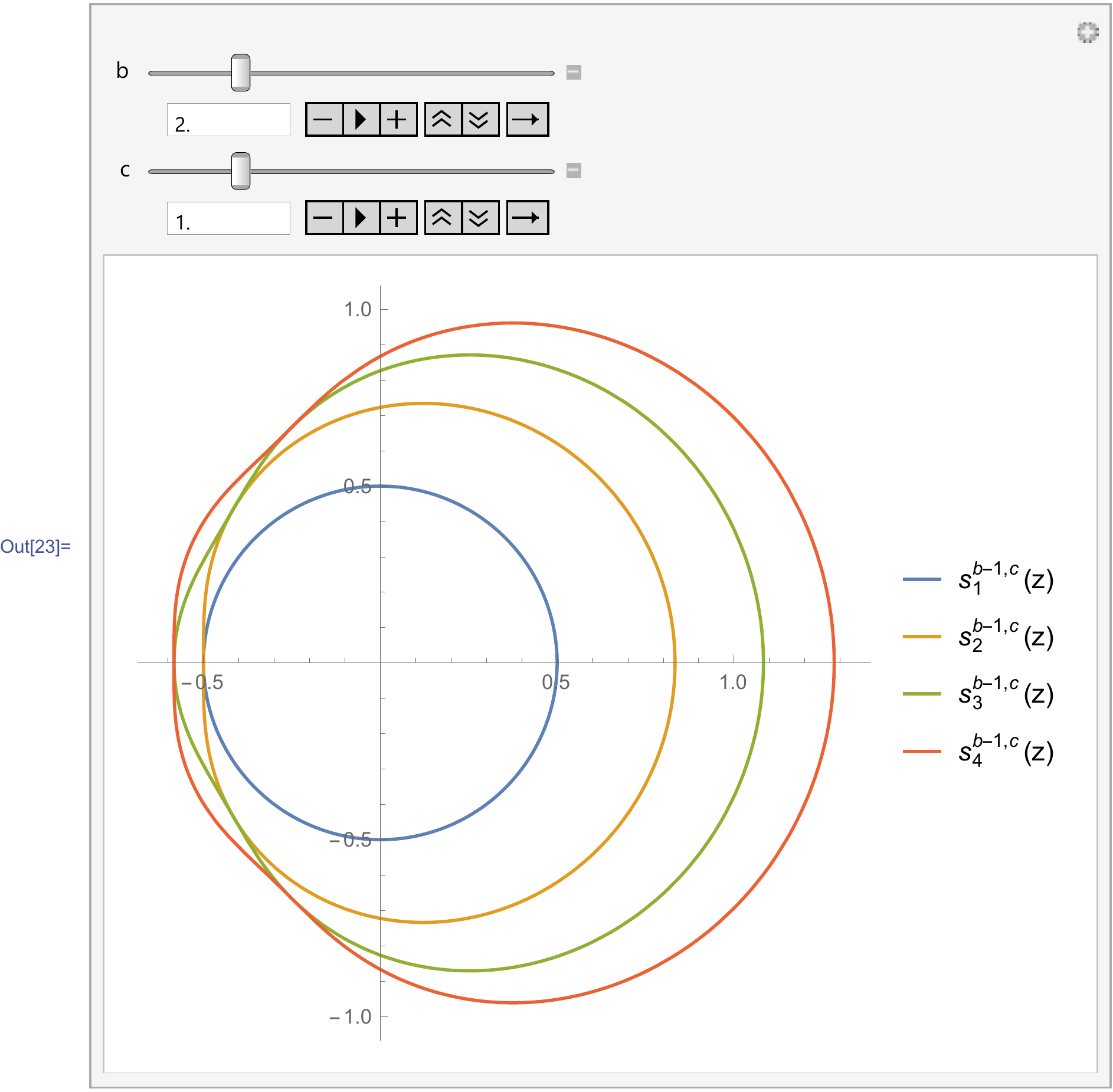}
\includegraphics[width=.4\columnwidth]{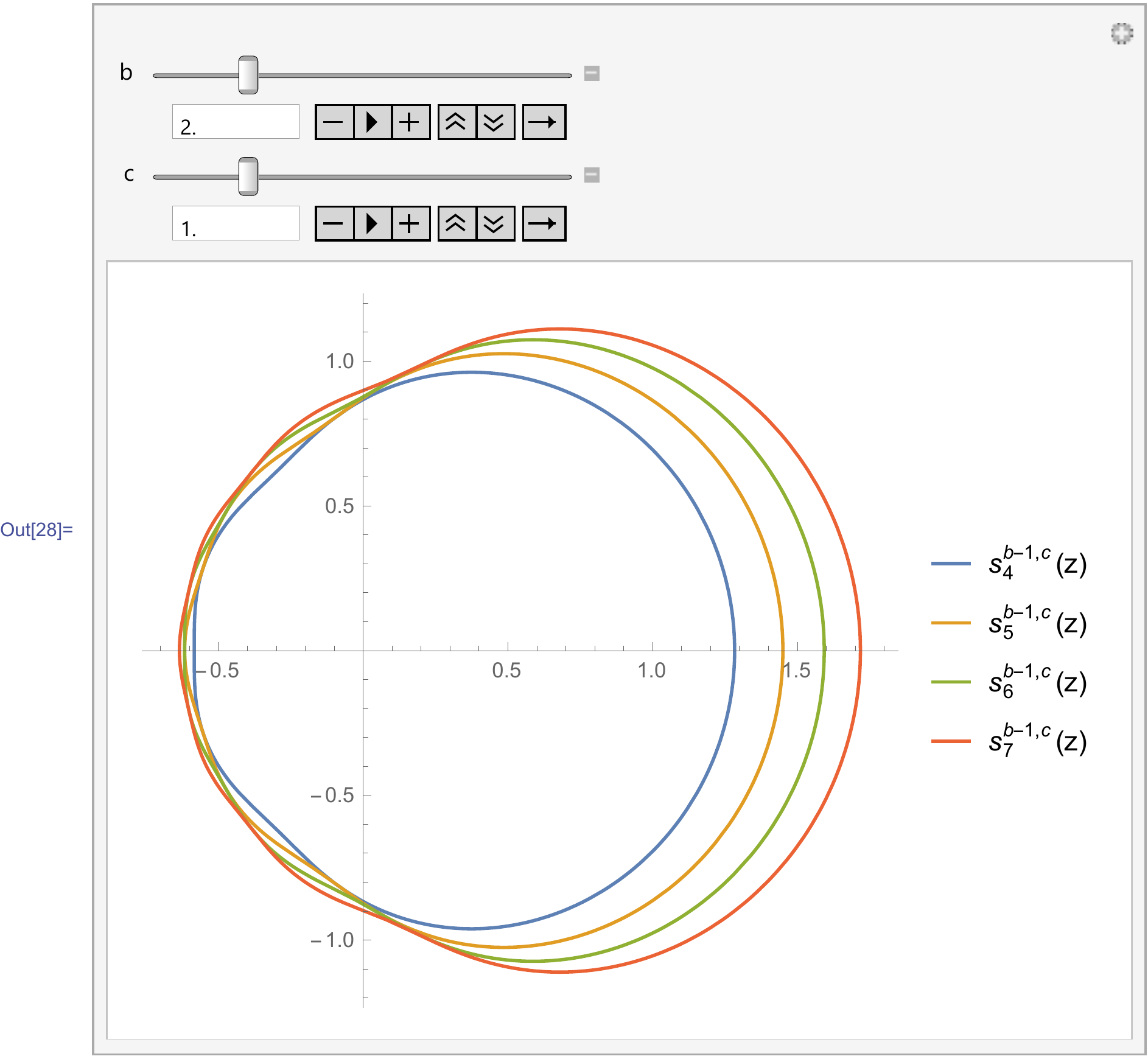}
\end{center}

\section{Concluding Remarks}\label{sec:consequences}
In this section, we define a set $\Omega$ be the set of nonnegative real numbers having the following property.
\begin{align*}
\Omega:=\{\mu_k\in[0,1]: \hbox{ such that } \sum_{k=1}^n\mu_k=1\}.
\end{align*}
In the context of generalization of Kakeya-Enestr\"om theorem given in \cite{ruscheweyh-1978-kakeya-thm-SIAM},
we have the following consequences of Theorem \ref{thm:starlike-generalized-stable}.

\begin{lemma} \cite{ruscheweyh-1978-kakeya-thm-SIAM}\label{lemma:ruscheweyh-kakeya-thm}
Let $n\in\mathbb{N}$ and $f(z)=z\displaystyle\sum_{k=0}^{\infty}b_kz^k\in\mathcal{S}^{\ast}(1/2)$.
Then $\exists$ a number $\rho=\rho(n,f)\geq 1$ such that for every sequence $a_k\in\mathbb{R},k=0,1,2\cdots,n,$ with
\begin{align*}
1=a_0\geq a_1\geq \cdots\geq a_n\geq 0,
\end{align*}
we have
\begin{align*}
P(z)=\sum_{k=0}^n a_kb_kz^k \neq 0, \quad \hbox{$|z|<\rho.$}
\end{align*}
\end{lemma}
We get the following consequences of Theorem \ref{thm:starlike-generalized-stable} using Lemma \ref{lemma:ruscheweyh-kakeya-thm}.
\begin{corollary}
Let $zf\in\mathcal{S}^{\ast}(\lambda),\lambda\in[1/2,1)$ and $b\geq\max\{c,2c-1\}$. Then for
any $\{\mu_k\}_{k=1}^n\in \Omega$, we have
\begin{align*}
\sum_{k=1}^n\mu_k \sigma_k^{(b-1,c)}(f,z)\neq 0,\quad \hbox{ $z\in\overline{\mathbb{D}}$}.
\end{align*}
\end{corollary}

\begin{proof}
Clearly $\{\mu_k\}_{k=1}^n\in\Omega$, implies $\displaystyle\sum_{k=1}^n\mu_k=1$. We consider
\begin{align*}
\sum_{k=1}^n \mu_k \sigma_k^{(b-1,c)}(f,z)=\sum_{k=0}^n \delta_k a_k z^k, \quad z\in\mathbb{D}.
\end{align*}
By simple calculation we can get that,
\begin{align*}
1=\delta_0\geq \delta_1 \geq \delta_2 \geq \cdots \geq \delta_n >0.
\end{align*}
Therefore $\delta_k$ satisfies the conditions of Lemma \ref{lemma:ruscheweyh-kakeya-thm}, hence we proved that
\begin{align}\label{eqn:mu-sk-nonzero}
\sum_{k=1}^n\mu_k \sigma_k^{(b-1,c)}(f,z)\neq 0,\quad \hbox{ $z\in\overline{\mathbb{D}}$}.\quad \quad  \qedhere
\end{align}
\end{proof}

Among several other consequences possible we would like to provide an application involving
Gegenbauer polynomials. Note that, for $0<\lambda<1/2$ and $-1\leq x\leq 1$,
\begin{align*}
G(z)=\frac{z}{(1-2xz+z^2)^{\lambda}}=z\sum_{k=0}^{\infty}C_k^{\lambda}(x)z^k \in\mathcal{S}^{\ast}(1-\lambda),
\end{align*}
where $C_k^{\lambda}$ are the Gegenbauer polynomial of degree $k$ and order $\lambda$. Therefore
(choosing $\mu_n=1$ and rest $\mu_k$ are all zero) we obtain,
\begin{align}
\sum_{k=0}^n \frac{B_{n-k}}{B_n}C_k^{\lambda}(x)z^k \neq 0, \quad z\in \mathbb{D}.
\end{align}
The inequality (5.1)
contains the result by Koumandos
\cite{koumanods-gegenbauer-chennai-2000} that the partial sum of
$G(z)/z$ i.e. $\sum_{k=0}^n C_k^{\lambda}(x)z^k$ are non-vanishing
in the closed unit disc for $0<\lambda<1/2$. This result enables us to
show that certain polynomials in $z$ having Gegenbauer polynomials
as a coefficients are zero free in the unit disc. This result will also be helpful
in proving positivity of Jacobi polynomial sums
\cite{lewis-1979-closetoconvexity-cesaro-mean-SIAM}. The inequality (5.1)
further can be sharpened in Corollary \ref{cor:5.2}.

\begin{corollary}\label{cor:5.2}
Let $zf\in\mathcal{S}^{\ast}(\lambda),\lambda\in[1/2,1)$ and $b\geq\max\{c,2c-1\}$.
Then for any $\{\mu_k\}_{k=1}^n\in \Omega$, we have
\begin{align*}
\left|\arg\sum_{k=1}^n\mu_k \sigma_k^{(b-1,c)}(f,z)\right| \leq 2\pi(1-\lambda),\quad \hbox{$z\in\overline{\mathbb{D}}$}.
\end{align*}
\end{corollary}
\begin{proof}
From Theorem \ref{thm:starlike-generalized-stable} we have
for $zf\in\mathcal{S}^{\ast}(\lambda),\lambda\in[1/2,1)$,
\begin{align*}
\sigma_n^{(b-1,c)}(f,z) = \left(\frac{1-\omega(z)}{1-z}\right)^{2-2\lambda},
\quad \hbox{where $|\omega(z)|\leq |z|$}.
\end{align*}
Choose $\mu_k,k=1,2,\ldots,n\in\Omega$ and taking the convex combination, we get
\begin{align*}
\sum_{k=1}^n \mu_k \sigma_k^{(b-1,c)}(f,z)=\left(\frac{1-\omega(z)}{1-z}\right)^{2-2\lambda}.
\end{align*}
This implies
\begin{align*}
\left|\arg \sum_{k=1}^n \mu_k \sigma_k^{(b-1,c)}(f,z)\right|&=(2-2\lambda)\left|\arg \left(\frac{1-\omega(z)}{1-z}\right)\right|\\
\Longrightarrow \left|\arg\sum_{k=1}^n\mu_k \sigma_k^{(b-1,c)}(f,z)\right| &\leq 2\pi(1-\lambda).
\quad \qedhere
\end{align*}
\end{proof}
\noindent Note that if $\lambda\in[3/4,1)$ and $zf\in\mathcal{S}^{\ast}(\lambda)$ then,
\begin{align*}
\left|\arg\sum_{k=1}^n\mu_k \sigma_k^{(b-1,c)}(f,z)\right|\leq \pi/2
\Longrightarrow
\mathrm{Re} \sum_{k=1}^n\mu_k \sigma_k^{(b-1,c)}(f,z)>0.
\end{align*}
Choose $\mu_n=1$ and rest of $\mu_k$ are zero.
\begin{align*}
\mathrm{Re}(\sigma_n^{(b-1,c)}(f,z))>0, \quad \hbox{$z\in\mathbb{D}$ and $n\in\mathbb{N}$.}
\end{align*}
Further in context of Gegenbauer polynomials this would imply for
$\lambda\in(0,1/4]$, $n\in\mathbb{N}$,
\begin{align}\label{eqn:gegenabuer-coro52}
\sum_{k=0}^n \frac{B_{n-k}}{B_n}C_k^{\lambda}(x) \cos k\theta>0,
\quad \theta\in(0,\pi), n\in\mathbb{N}.
\end{align}
This estimate of the upper bound on $\lambda$ in \eqref{eqn:gegenabuer-coro52}
is not sharp. The theory of starlike functions
ensure that the upper bound will be evaluated at $x=1$ for the large values of $n$.
However, for the case $b=c=1$, this problem was solved by Koumandos and Ruscheweyh
\cite{koumandos-ruscheweyh-2006-gegenbauer-Const.Approx}. For that case,
the upper bound for $\lambda$ is $\lambda=0.345778\ldots$. In general to
find the upper bound for $\lambda$, for values of $b$ and $c$, will lead
to new problem which will have further implications.

\textbf{Acknowledgement:}{ The first author is thankful to the
"Council of Scientific and Industrial Research, India"
(grant code: 09/143(0827)/2013-EMR-1) for financial support to carry out
the above research work.
}

\end{document}